\documentclass{amsart}

\newtheorem{theorem}{Theorem}[section]
\newtheorem{lemma}[theorem]{Lemma}
\newtheorem{cor}[theorem]{Corollary}
\usepackage[dvipdfm]{graphicx,color}
\usepackage{algorithm}
\usepackage{algorithmic}
\usepackage{multirow,bigdelim}
\usepackage{amsmath}
\usepackage{bm}
\usepackage{amssymb}
\usepackage{amscd}
\usepackage{amsfonts} 

\theoremstyle{definition}
\newtheorem{definition}[theorem]{Definition}

\theoremstyle{remark}

\numberwithin{equation}{section}

\begin{document}

\title[Gerschgorin's theorem for generalized
    eigenvalue problems]{Gerschgorin's theorem for generalized
    eigenvalue problems in the Euclidean metric}
\author{Yuji Nakatsukasa}
\address{Department of Mathematics, University of California, Davis}
\curraddr{
   Department of Mathematics, One Shields Avenue, Davis, California, 95616
}
\email{ynakam@math.ucdavis.edu}



\date{\today}


\begin{abstract} 
We present Gerschgorin-type eigenvalue inclusion sets  applicable to generalized eigenvalue problems. 
Our sets are defined by circles in the complex plane in the standard Euclidean metric, and are easier to compute than known similar results. 
As one application we use our results  to provide a forward   error analysis for a computed eigenvalue of a  diagonalizable pencil.
\end{abstract}

\maketitle

%

\section{Introduction}
For a standard eigenvalue problem $A\bm{x}=\lambda\bm{x}$ where $A\in \mathbb{C}^{n\times n}$,  Gerschgorin's theorem \cite{gersch} defines  in the complex plane a union of $n$ disks that contains all the $n$ eigenvalues. Its simple exposition and applicability make it an extremely  useful tool in estimating eigenvalue bounds.  
It also plays an important role  in eigenvalue perturbation theory \cite{wilkinson:1965,stewart-sun:1990}.


The generalized eigenvalue problem $A\bm{x}=\lambda B\bm{x}$  where $A,B\in \mathbb{C}^{n\times n}$ also arises in many scientific applications. It should be useful to have available a similar simple theory to estimate the eigenvalues for this type of problems as well.

In fact, Stewart and Sun \cite{Stewart75,stewart-sun:1990} provide an eigenvalue inclusion set applicable to generalized eigenvalue problems. 
The set is the union of $n$ regions defined by
\begin{equation}\label{steg}
G_i(A,B)\equiv \left\{z \in \mathbb{C}:\chi(z, a_{i,i}/b_{i,i})\leq  \varrho_i\right\},
\end{equation}
where 
\begin{equation}\label{rostewart}
\varrho_i=\sqrt{\frac{\left(\sum_{j\not= i}|a_{i,j}|\right)^2+\left(\sum_{j\not= i}|b_{i,j}|\right)^2}{|a_{i,i}|^2+|b_{i,i}|^2}}.
\end{equation}
All the eigenvalues of the pencil $A-\lambda B$ lie in the union of $G_i(A,B)$, i.e., if $\lambda$ is an eigenvalue, then 
\[ 
\lambda \in  G(A,B) \equiv\bigcup_{i=1}^n G_i(A,B).  
\]
Note that $\lambda$ can be infinite. We briefly review the  definition of eigenvalues of a pencil  at the beginning of section \ref{ggsec}. 

 The region \eqref{steg} is defined in terms of the chordal metric $\chi$, defined by \cite[Ch.7.7]{Golubbook}
\begin{equation}\nonumber
\chi(x,y)=\frac{|x-y|}{\sqrt{1+|x|^2}\sqrt{1+|y|^2}}.
\end{equation}
 The justification of using the chordal metric instead of the more standard Euclidean metric is in the unifying treatment of finite and infinite eigenvalues  \cite{Stewart75,stewart-sun:1990}. The use of the chordal metric has thus become a common practice in perturbation analyses for generalized eigenvalue problems, and some recent   results \cite{xiao07,li02} are  presented in terms of this metric. 

 However, using the chordal metric makes the application of the theory  less intuitive and usually more complicated. In particular, interpreting the set $G$ in the Euclidean metric is a difficult task, as opposed to the the Gerschgorin set for standard eigenvalue problems, which is defined as a union of $n$ disks. 
Another caveat of using $G$ is that it is not clear whether the region $G$ will give a nontrivial estimate of the eigenvalues. Specifically, since any two points in the complex plane have distance smaller than 1 in the chordal metric, if there exists $i$ such that $\varrho_i\geq 1$, then  $G$ is the whole complex plane, providing no information. 
In view of \eqref{rostewart}, it follows that  $G$ is useful only when both $A$ and $B$ have small off-diagonal elements. 

Another Gerschgorin-type eigenvalue localization theory applicable to generalized eigenvalue problems appear in a recent paper \cite{gengervarga} by Kostic et al. Their inclusion set is defined by 
\begin{equation}
  \label{vargager}
K_i(A,B)\equiv \left\{z \in \mathbb{C}:|b_{i,i}z-a_{i,i}|\leq \sum_{j\not= i}|b_{i,j}z-a_{i,j}|\right\},
\end{equation}
and all the eigenvalues of the pencil $A-\lambda B$ exist in the union $K(A,B) \equiv\bigcup_{i=1}^n K_i(A,B)$. This set is defined in the Euclidean metric, and \eqref{vargager} shows that $K(A,B)$ is a compact set in the complex plane $\mathbb{C}$ if and only if $B$ is strictly diagonally dominant.
However, the set \eqref{vargager} is in general a complicated region, which makes its practical application difficult. 


The goal of this paper is to  present a different generalization of  Gerschgorin's theorem applicable to generalized eigenvalue problems, which solves the issues mentioned above. In brief, our eigenvalue inclusion sets  have the following properties:
\begin{itemize}
\item They involve only circles in the Euclidean complex plane,  using the same information as \eqref{steg} does. Therefore it is simple to compute and visualize. 
\item They are defined in the Euclidean metric, but still deal with finite and infinite eigenvalues uniformly.
\item One variant $\Gamma^S(A,B)$ is a union of $n$ disks when $B$ is strictly diagonally dominant. 
\item Comparison with $G(A,B)$: Our results are defined in the Euclidean metric. Tightness is incomparable, but our results are tighter when $B$ is close to a diagonal matrix.
\item Comparison with $K(A,B)$: Our results are defined by circles and are much simpler. $K(A,B)$ is always tighter, but our results approach $K(A,B)$ when $B$ is close to a diagonal matrix.
\end{itemize}
In summary, our results provide a method for estimating eigenvalues of $(A,B)$ in a much cheaper way than the two known results do. 

 The structure of the paper is  as follows. In section \ref{ggsec} we describe our idea and derive our main Gerschgorin theorems for generalized eigenvalue problems. Simple examples and plots are shown in section \ref{ex} to illustrate the properties of different regions.  
Section 4 presents one application of our results, where  we develop  forward error analyses for the computed eigenvalues of a non-Hermitian generalized eigenvalue problem. 
\section{Main  Gerschgorin theorems}\label{ggsec}
 In this section we develop our Gerschgorin theorem and its  variants. First the basic idea for bounding the  eigenvalue location is  discussed. Section \ref{gersec1} presents a  simple bound and derives our first Gerschgorin theorem. In section \ref{gersec2} we carry out a more careful analysis and obtain  a tighter result. In section \ref{gersec3} we show that our results can  localize a specific number of  eigenvalues, a well-known property of $G$ and the  Gerschgorin set for standard eigenvalue problems. 

As a brief summary of the eigenvalues of a pencil  $A-\lambda B$ where $A,B\in \mathbb{C}^{n\times n}$, 
$\lambda$ is a finite  eigenvalue of the pencil if $\mbox{det}(A-\lambda B)=0$, and in this case there exists nonzero $\bm{x}\in \mathbb{C}^{n}$ such that $A\bm{x}=\lambda B\bm{x}$. 
 If the degree of the characteristic polynomial $\mbox{det}(A-\lambda B)$ is $d<n$, then we say the pencil has $n-d$ infinite eigenvalues. In this case, there exists a nonzero vector $\bm{x}\in \mathbb{C}^{n}$ such that $ B\bm{x}=0$. 
When $B$ is nonsingular, the pencil has $n$ finite eigenvalues, matching those of $B^{-1}A$. 

Throughout the paper we assume that for each $i\in \{1,2,\cdots, n\}$, the $i$th row of either $A$ or $B$ is strictly diagonally dominant, unless otherwise mentioned. Although this may seem a rather restrictive assumption, its justification is the  observation that the set $G(A,B)$ is always the entire complex plane unless this assumption  is true.

 \subsection{Idea}\label{first}

Suppose $A\bm{x}=\lambda B\bm{x}$ (we consider the case $\lambda=\infty$ later).  We write $\bm{x}=(x_1,x_2,\cdots, x_n)^T$ and denote by $a_{p,q}$ and $b_{p,q}$the ($p,q$)th element of $A$ and $B$ respectively. Denote by $i$ the integer such that $|x_i|=\max_{1\leq j\leq n}|x_j|$, so that $x_i\not=0$. First  we consider the case where the  $i$th row of $B$ is strictly diagonally dominant, so   $ |b_{i,i}|>\sum_{j\not= i}|b_{i,j}|$.
From the  $i$th equation  of $A\bm{x}=\lambda B\bm{x}$ we have
\begin{align}\label{starthere}
a_{i,i}x_i+\sum_{j\not= i}a_{i,j}x_j=\lambda(b_{i,i}x_i+\sum_{j\not= i}b_{i,j}x_j).
\end{align}
Dividing both sides by $x_i$ and rearranging yields 
\[
\lambda\left( b_{i,i}+\sum_{j\not= i}b_{i,j}\frac{x_j}{x_i}\right)-a_{i,i}=\sum_{j\not= i}a_{i,j}\frac{x_j}{x_i}.
\]
\begin{equation}\label{kokokara}
\therefore \left|\lambda \left(b_{i,i}+\sum_{j\not= i}b_{i,j}\frac{x_j}{x_i}\right)-a_{i,i}\right|\leq \sum_{j\not= i}|a_{i,j}|\frac{|x_j|}{|x_i|}\leq R_i,
\end{equation}
where we write $ R_i=\sum_{j\not= i}|a_{i,j}|$. The last inequality holds because $|x_j|\leq |x_i|$ for all $j$.
 Here, using the assumption  $|b_{i,i}|>\sum_{j\not= i}|b_{i,j}|$, we have 
\begin{align*}
\left|b_{i,i}+\sum_{j\not= i}b_{i,j}\frac{x_j}{x_i}\right|&\geq |b_{i,i}|-\sum_{j\not= i}|b_{i,j}|\frac{|x_j|}{|x_i|}\\
&\geq |b_{i,i}|-\sum_{j\not= i}|b_{i,j}|>0,
\end{align*}
 where we  used   $|x_j|\leq |x_i|$ again.  Hence we can  divide \eqref{kokokara} by $\left|b_{i,i}+\sum_{j\not= i}b_{i,j}\frac{x_j}{x_i}\right|$, which yields
\begin{equation}\label{kok}
 \left|\lambda -\frac{a_{i,i}}{\left(b_{i,i}+\sum_{j\not= i}b_{i,j}\frac{x_j}{x_i}\right)}\right|\leq \frac{R_i}{\left|b_{i,i}+\sum_{j\not= i}b_{i,j}\frac{x_j}{x_i}\right|}.
\end{equation}
 Now, writing $\gamma_i=(\sum_{j\not= i}b_{i,j}\frac{x_j}{x_i})/b_{i,i}$, we have $|\gamma_i|\leq \sum_{j\not= i}|b_{i,j}|/|b_{i,i}|\  (\equiv r_i)< 1$, 
and \eqref{kok} becomes
\begin{equation}\label{kok2}
 \left|\lambda -\frac{a_{i,i}}{b_{i,i}}\cdot \frac{1}{1+\gamma_i}\right|\leq \frac{R_i}{|b_{i,i}|}\frac{1}{|1+\gamma_i|}.
\end{equation}
Our interpretation of  this inequality is as follows:   $\lambda$ lies in the disk  of radius $R_i/|b_{i,i}||1+\gamma_i|$  centered at  $a_{i,i}/b_{i,i}(1+\gamma_i)$, defined in the complex plane. Unfortunately the exact value of  $\gamma_i$ is unknown, so we cannot specify the disk. Fortunately, we show in section \ref{gersec1} that using  $|\gamma_i|\leq  r_i$ we can obtain a region that contains all the disks defined by  \eqref{kok2} for any $\gamma_i$ such that $|\gamma_i|\leq r_i$.  

 Before we go on to analyze the inequality \eqref{kok2}, let us consider the case where the $i$th row of $A$ is strictly diagonally dominant. As we will see, this  also lets us treat infinite eigenvalues. 


Recall  \eqref{starthere}.  We first note that  if  $|x_i|=\max_{j}|x_j|$ and  the $i$th row of $A$ is strictly diagonally dominant, then $\lambda\not= 0$, 
because $|a_{i,i}x_i+\sum_{j\not= i}a_{i,j}x_j|\geq |a_{i,i}||x_i|-\sum_{j\not= i}|a_{i,j}||x_j|\geq |x_i|(|a_{i,i}|-\sum_{j\not= i}|a_{i,j}|)>0$. Therefore, in place of  \eqref{starthere} we  start with the equation 
\begin{align*}
b_{i,i}x_i+\sum_{j\not= i}b_{i,j}x_j=\frac{1}{\lambda}\left(a_{i,i}x_i+\sum_{j\not= i}a_{i,j}x_j\right).
\end{align*}
 Note that this expression includes the case $\lambda=\infty$, because then the equation becomes $B\bm{x}=0$. 
Following the same analysis as above, we arrive at  the  inequality corresponding to  \eqref{kok2}:
\begin{equation}\label{kok2a}
 \left|\frac{1}{\lambda} -\frac{b_{i,i}}{a_{i,i}}\cdot \frac{1}{1+ \gamma_i^A}\right|\leq \frac{R_i^A}{|a_{i,i}|}\frac{1}{|1+\gamma^A_i|},
\end{equation}
where we write  $ R_i^A=\sum_{j\not= i}|b_{i,j}|$ and  $\gamma_i^A=(\sum_{j\not= i}a_{i,j}\frac{x_j}{x_i})/a_{i,i}$. Note that \\
  $|\gamma_i^A|\leq \sum_{j\not= i}|a_{i,j}|/|a_{i,i}|\  (\equiv r_i^A)< 1$. Therefore we are in an essentially  same situation as in  \eqref{kok2}, the only difference being  that we are bounding  $1/\lambda$ instead of $\lambda$. 

In summary, in both cases the problem  boils down to finding a region that contains all $z$ such that 
\begin{equation}\label{inz}
 \left|z - \frac{s}{1+\gamma}\right|\leq \frac{t}{|1+\gamma|},
\end{equation}
 where $s\in \mathbb{C},t>0$ are known and $0<r<1$ is known such that  $|\gamma|\leq r$. 
\subsection{Gerschgorin theorem}\label{gersec1}
First we bound the right-hand side of  \eqref{inz}. This can be done simply by 
\begin{equation}
  \label{radius}
\frac{t}{|1+\gamma|}\leq \frac{t}{1-|\gamma|}\leq\frac{t}{1-r}.  
\end{equation}

Next we consider a region that contains all the possible centers of the disk \eqref{kok2}.  We use the following result.
\begin{lemma}\label{rbound}
If $|\gamma|\leq r<1$, then the point $1/(1+\gamma)$ lies in the disk in the complex plane of  radius $r/(1-r)$ centered at $1$.
\end{lemma}

\begin{proof}
\begin{align*}
\left|\frac{1}{1+\gamma}-1\right|&=\left|\frac{\gamma}{1+\gamma}\right|\\
&\leq\frac{r}{1-|\gamma|}\\
& \leq \frac{r}{1-r}.
\end{align*}
\end{proof}
In view of  \eqref{inz}, this means that  $s/(1+\gamma)$,  the center of the disk \eqref{inz}, has to lie in the disk of  radius $sr/(1-r)$ centered at $\displaystyle s$. 
Combining this and \eqref{radius},  we conclude  that  $z$ that satisfies \eqref{inz} is included in the disk of radius $\displaystyle\frac{sr}{1-r}+\frac{t}{1-r}$, centered at $s$. 

Using this for \eqref{kok2} by letting $s=|a_{i,i}|/|b_{i,i}|, t=R_i/|b_{i,i}|$ and $r=r_i$,  we see that $\lambda$ that satisfies \eqref{kok2}  is necessarily included in the disk centered at $a_{i,i}/b_{i,i}$, and of radius 
\[\rho_i=\frac{|a_{i,i}|}{|b_{i,i}|}\frac{r_i}{1-r_i}+ \frac{R_i}{|b_{i,i}|}\frac{1}{1-r_i}= \frac{|a_{i,i}|r_i+R_i}{|b_{i,i}|(1-r_i)}.\]
 Similarly, applying the result to \eqref{kok2a}, we see that $1/\lambda$ satisfying  \eqref{kok2a} has to satisfy 
\begin{equation}\label{kok2a2}
 \left|\frac{1}{\lambda} -\frac{b_{i,i}}{a_{i,i}}\right|\leq \frac{|b_{i,i}|r^A_i+R_i^A}{|a_{i,i}|(1-r^A_i)}.
\end{equation}
This is equivalent to 
\[
 \left|a_{i,i}-\lambda b_{i,i}\right|\leq \frac{|b_{i,i}|r^A_i+R_i^A}{(1-r^A_i)}|\lambda|.
\]

If $b_{i,i}=0$, this becomes $\displaystyle  \frac{R_i^A}{(1-r^A_i)}|\lambda|\geq |a_{i,i}|$, which is $\displaystyle |\lambda|\geq  \frac{|a_{i,i}|}{R_i^A}(1-r^A_i)$ when $R_i^A\not=0$. If $b_{i,i}=R_i^A=0$, no finite $\lambda$ satisfies the inequality, so we say the point $\lambda=\infty$  includes  the inequality. 

If $b_{i,i}\not= 0$, we have
 \begin{equation}\label{apolo}
\left|\lambda -\frac{a_{i,i}}{b_{i,i}}\right|\leq  \frac{|b_{i,i}|r^A_i+R_i^A}{|b_{i,i}|(1-r^A_i)}\cdot |\lambda |.   
\end{equation}
 For simplicity, we write this inequality as 
 \begin{equation}
   \label{simple}
|\lambda-\alpha_i|\leq \beta_i |\lambda|,   
 \end{equation}
 where $\displaystyle \alpha_i=\frac{a_{i,i}}{b_{i,i}}$ and $\displaystyle\beta_i= \frac{|b_{i,i}|r^A_i+R_i^A}{|b_{i,i}|(1-r^A_i)}>0$. 
Notice that the equality of  \eqref{simple} holds on a certain  circle of Apollonius \cite[sec.2]{geometry}, defined by $|\lambda-\alpha_i|= \beta_i |\lambda|$. It is easy to see that the radius of the Apollonius circle is 
\[\rho^A_i=\left|\frac{1}{2}\left(\frac{|\alpha_i|}{1-\beta_i}-\frac{|\alpha_i|}{1+\beta_i}\right)\right|=\frac{|\alpha_i|\beta_i}{|1-\beta_i^2|},\]
 and the center is 
\[c_i=\frac{1}{2}\left(\frac{\alpha_i}{1+\beta_i}+\frac{\alpha_i}{1-\beta_i}\right)=\frac{\alpha_i}{1-\beta_i^2}.\]
From  \eqref{simple} we observe the following. The Apollonius circle divides the complex plane into two regions, and    $\lambda$ exists in the region that contains $\alpha_i=a_{i,i}/b_{i,i}$. 
Consequently,  $\lambda$ lies outside the circle of Apollonius when $\beta_i>1$, and inside it when $\beta_i<1$. When $\beta_i=1$, 
the Apollonius circle is  the perpendicular bisector of the line that connects $\alpha_i$ and $0$, dividing the  complex plane into halves. 

 The above arguments motivate the following definition. 
 \begin{definition}\label{def1}
For   $n\times n$ complex  matrices $A$ and $B$, denote by $S^B$ (and $S^A$) the set of $i\in \{1,2,\cdots, n\}$ such that the $i$th row of $B$ ($A$) is
 strictly diagonally dominant. 

For $i\in S^B$, define the disk $\Gamma_i^B(A,B)$ by
\begin{equation}\label{mygb}
\Gamma_i^B(A,B)\equiv\left\{z \in \mathbb{C}:\left|z-\frac{a_{i,i}}{b_{i,i}}\right|\leq \rho_i \right\} 
\quad (i=1, 2, \cdots, n),
\end{equation} 
where denoting $\displaystyle r_i =\sum_{j\not= i}\frac{|b_{i,j}|}{|b_{i,i}|}(<1)$ and $\displaystyle R_i=\sum_{j\not= i}|a_{i,j}|$, the radii $\rho_i $ are defined by
\[\rho_i = \frac{|a_{i,i}|r_i +R_i}{|b_{i,i}|(1-r_i )}.\]
For $i\notin S^B$, we set $\Gamma_i^B(A,B)=\mathbb{C}$, the whole complex plane. 

We also define $\Gamma_i^A(A,B)$ by the following.  For $i\in S^A$,  
denote $\displaystyle r_i^A=\sum_{j\not= i}\frac{|a_{i,j}|}{|a_{i,i}|}(<1)$ and $\displaystyle R_i^A=\sum_{j\not= i}|b_{i,j}|$. 

If $b_{i,i}=R^A_i=0$, define $\Gamma_i^A(A,B)=\{\infty\}$, the point $z=\infty$.
 If $b_{i,i}=0$ and $R^A_i>0$, define $\Gamma_i^A(A,B)\equiv\left\{z \in \mathbb{C}:|z|\geq  \frac{|a_{i,i}|}{R_i^A}(1-r^A_i)\right\}$. 

For $b_{i,i}\not= 0$, denoting $\displaystyle \alpha_i=\frac{a_{i,i}}{b_{i,i}}$ and $\displaystyle\beta_i= \frac{|b_{i,i}|r^A_i+R_i^A}{|b_{i,i}|(1-r^A_i)}$, \\
\begin{itemize}
\item If $\beta_i<1$, then define
\begin{equation}\label{myg2}
\Gamma_i^A(A,B)\equiv\left\{z \in \mathbb{C}:\left|z-c_i\right|\leq \rho_i^A\right\},
\end{equation} 
where $\displaystyle c_i=\frac{\alpha_i}{1-\beta_i^2}$ and $\displaystyle \rho_i^A=\frac{|\alpha_i|\beta_i}{|1-\beta_i^2|}$. 
\item If $\beta_i>1$, then define
\begin{equation}\label{myginf}
\Gamma_i^A(A,B)\equiv\left\{z \in \mathbb{C}:\left|z-c_i\right|\geq  \rho_i^A\right\},
\end{equation} 
\item If $\beta_i=1$, then define
\begin{equation}\label{myghalf}
\Gamma_i^A(A,B)\equiv\left\{z \in \mathbb{C}:|z-\alpha_i|\leq |z|\right\}. 
\end{equation} 

\end{itemize}


Finally for $i\notin S^A$, we set $\Gamma_i^A(A,B)=\mathbb{C}$. 
\end{definition}
\bigskip

Note that $\Gamma_i^A(A,B)$ in \eqref{myginf} and \eqref{myghalf} contains the point $\{\infty\}$. 

We now present our eigenvalue localization theorem.

\begin{theorem}[Gerschgorin-type theorem for generalized eigenvalue problems]\label{gger}
 Let $A,B$ be $n\times n$ complex  matrices. 

All the eigenvalues of the pencil $A-\lambda B$  lie in the  union of $n$ regions $\Gamma_i(A,B)$ in the complex plane defined by
\begin{equation}\label{myg}
\Gamma_i(A,B)\equiv  \Gamma_i^B(A,B)\cap \Gamma_i^A(A,B).
\end{equation} 
 In other words, if $\lambda$ is an eigenvalue of the pencil, then 
\[\lambda\in \Gamma(A,B)\equiv\bigcup_{1\leq i\leq n}\Gamma_i(A,B).\]
\end{theorem}
\begin{proof}
First consider the case where $\lambda$ is a  finite  eigenvalue, so that $A\bm{x}=\lambda B\bm{x}$. 
The above arguments show that $\lambda\in  \Gamma_i(A,B)$ for $i$ such that $|x_i|=\max_j |x_j|$. 

Similarly, in the infinite eigenvalue case $\lambda=\infty$,  
let $B\bm{x}=0$. Note that  the $i$th row (such that $|x_i|=\max_j |x_j|$) of $B$ cannot be strictly diagonally dominant, because if it is, then  $|b_{i,i}x_i+\sum_{j\not= i}b_{i,j}x_j|\geq |b_{i,i}||x_i|-\sum_{j\not= i}|b_{i,j}||x_j|\geq |x_i|(|b_{i,i}|-\sum_{j\not= i}|b_{i,j}|)>0$. Therefore, $\Gamma_i^B(A,B)=\mathbb{C}$, so $\Gamma_i(A,B)=\Gamma_i^A(A,B)$. Here if $i\notin S^A$, then $\Gamma_i(A,B)=\mathbb{C}$, so $\lambda \in \Gamma(A,B)$ is trivial. Therefore we consider the case $i\in S^A$. Note that the fact that $B$ is not strictly diagonally dominant implies $|b_{i,i}|<R_i^A$, which in turn means $\beta_i>1$, because recalling that $\displaystyle\beta_i= \frac{|b_{i,i}|r^A_i+R_i^A}{|b_{i,i}|(1-r^A_i)}$, we have
\[
|b_{i,i}|r^A_i+R_i^A-|b_{i,i}|(1-r^A_i)=|b_{i,i}|(2r^A_i-1)+R_i^A>2r^A_i|b_{i,i}|>0. 
\]
Hence, recalling \eqref{myginf} we see that ${\infty}\in \Gamma_i^A(A,B)$. 


Therefore,  any eigenvalue of the pencil  lies in  $\Gamma_i(A,B)$ for some $i$, so all the eigenvalues lie in the union $ \bigcup_{1\leq i\leq n}\Gamma_i(A,B)$. 
\end{proof}


 Theorem \ref{gger} shares the properties with the standard Gerschgorin theorem that it is an eigenvalue inclusion set that is easy to compute, and the boundaries are defined as circles (except for $\Gamma_i^A(A,B)$ for the special case $\beta_i=1$). One difference between the two is that Theorem \ref{gger} involves $n+m$ circles, where $m$  is the number of rows for which both $A$ and $B$ are strictly diagonally dominant. By contrast, the standard Gerschgorin always needs $n$ circles. Also, when $B\rightarrow I$, the set does not become the standard Gerschgorin set, but rather becomes a slightly tighter set (owing to $\Gamma_i^A(A,B)$). 
 Although these are not serious defects of out set $\Gamma(A,B)$, the following  simplified  variant  solves the two issues. 
 \begin{definition}
We use  the notations in Definition \ref{def1}. 
For $i\in S^B$, define $\Gamma_i^S(A,B)$ by $\Gamma_i^S(A,B)=\Gamma_i^B(A,B)$. For  $i\notin S^B$, define $\Gamma_i^S(A,B)=\Gamma_i^A(A,B)$.   
 \end{definition}
\begin{cor}\label{simp}
 Let $A,B$ be $n\times n$ complex  matrices. All the eigenvalues of the pencil $A-\lambda B$  lie in $\displaystyle\Gamma^S(A,B)=\bigcup_{1\leq i\leq n}\Gamma_i^S(A,B)$. 
\end{cor}
\begin{proof}
 It is easy to see that $\Gamma_i(A,B)\subseteq \Gamma_i^S(A,B)$ for all $i$. Using Theorem \ref{gger} the conclusion follows immediately. 
\end{proof}
As a special case, this result becomes a union of $n$ disks when $B$ is strictly diagonally dominant. 
\begin{cor}\label{simp2}
 Let $A,B$ be $n\times n$ complex  matrices, and let  $B$ be strictly diagonally dominant. 
 Then, $\Gamma_i^S(A,B)=\Gamma_i^B(A,B)$, and 
  denoting by $\lambda_1,\cdots,\lambda_n$ the $n$ finite eigenvalues of the pencil $A-\lambda B$, 
\[\lambda\in \Gamma^S(A,B)=\bigcup_{1\leq i\leq n}\Gamma_i^B(A,B).\]
\end{cor}
\begin{proof}
  The fact that $\Gamma_i^S(A,B)=\Gamma_i^B(A,B)$ follows immediately from the diagonal dominance of $B$. The diagonal dominance of $B$ also forces it to be nonsingular, so that the pencil $A-\lambda B$ has $n$ finite eigenvalues. 
\end{proof}

Several points are worth noting regarding the above results.
\begin{itemize}
\item $\Gamma^S(A,B)$ in Corollaries \ref{simp} and \ref{simp2} is defined by $n$ circles. Moreover, it is easy to see that  $\Gamma^S(A,B)$ reduces to the original Gerschgorin theorem by letting $B=I$. In this respect $\Gamma^S(A,B)$ might be considered a more natural generalization of the standard Gerschgorin theorem than $\Gamma(A,B)$. We note that these properties are shared by $K(A,B)$ in \eqref{vargager} but not shared by $G(A,B)$ in \eqref{steg}, which is defined by $n$ regions, but not circles in the Euclidean metric, and is not equivalent  to (always worse, see below) the standard Gerschgorin set when $B=I$. 
 $\Gamma(A,B)$ also shares with $K(A,B)$ the property that it is a compact set in $\mathbb{C}$ if and only if $B$ is strictly diagonally dominant, as mentioned in Theorem 8 in \cite{gengervarga}. 
\item $K(A,B)$ is always included in $\Gamma(A,B)$. 
To see this, suppose that $z\in K_i(A,B)$ so $|b_{i,i}z-a_{i,i}|\leq \sum_{j\not= i}|b_{i,j}z-a_{i,j}|$. Then for $b_{i,i}\not= 0$, (note that $\Gamma_i^B(A,B)=\mathbb{C}$ so trivially $z\in \Gamma_i^B(A,B)$ if $b_{i,i}=0$)
\begin{align*}
&|b_{i,i}z-a_{i,i}|\leq \sum_{j\not= i}|b_{i,j}z|+\sum_{j\not= i}|a_{i,j}|  \\
\Leftrightarrow &|z-\frac{a_{i,i}}{b_{i,i}}|-\sum_{j\not= i}\frac{|b_{i,j}z|}{|b_{i,i}|}\leq \frac{R_i}{|b_{i,i}|}\quad \left(\mbox{recall}\ R_i=\sum_{j\not= i}|a_{i,j}|\right)\\
\Rightarrow &|z-\frac{a_{i,i}}{b_{i,i}}|-r_i|z|\leq \frac{R_i}{|b_{i,i}|}. \quad \left(\mbox{recall}\  r_i=\sum_{j\not= i}\frac{|b_{i,j}|}{|b_{i,i}|}\right)
\end{align*}
Since we can write 
$|z-a_{i,i}|-r_i|z|=|z-a_{i,i}+ r_ie^{i\theta}z|$ for some $\theta\in [0,2\pi]$, it follows that if $z\in K_i(A,B)$ then 
\[\left|z(1+r_ie^{i\theta})-\frac{a_{i,i}}{b_{i,i}}\right|\leq \frac{R_i}{|b_{i,i}|}. \]
Since $r_i<1$, we can divide this by $(1+r_ie^{i\theta})$, which yields

\begin{equation}
  \label{zin}
\left|z-\frac{a_{i,i}}{b_{i,i}}\frac{1}{1+r_ie^{i\theta}}\right|\leq \frac{R_i}{|b_{i,i}|}\frac{1}{|1+r_ie^{i\theta}|}.  
\end{equation}
Note that this becomes \eqref{kok2} if we substitute  $\gamma_i$ into $r_ie^{i\theta}$ and $\lambda$ into $z$. 
Now, since $\Gamma_i^B(A,B)$ is derived from \eqref{kok2} by considering a disk that contains $\lambda$ that satisfies \eqref{kok2} for any $\gamma_i$ such that $|\gamma_i|<r_i$, it follows that $z$ that satisfies \eqref{zin} is included in  $\Gamma_i^B(A,B)$. By a similar argument we can prove $z\in K_i(A,B)\Rightarrow z\in \Gamma_i^A(A,B)$, so the claim is proved. 
\item Although $K(A,B)$ is always sharper than $\Gamma(A,B)$ is,  $\Gamma(A,B)$ has the obvious advantage over $K(A,B)$ in its practicality. $\Gamma(A,B)$ is much easier to compute than $K(A,B)$, which is generally a union of complicated regions. It is also easy to see that $\Gamma(A,B)$ approaches $K(A,B)$ as $B$ approaches a diagonal matrix, see examples in section \ref{ex}.  $\Gamma(A,B)$ sacrifices some tightness for the sake of simplicity.
For instance, $K(A,B)$ is difficult to use for the analysis in section \ref{forerr}.
\item $G(A,B)$ and $\Gamma(A,B)$ are generally not comparable, see the examples in section \ref{ex}. 
However, 
we can see that $\Gamma_i(A,B)$ is a nontrivial set in the complex plane  $\mathbb{C}$ whenever $G_i(A,B)$ is, but the contrary does not hold. This can be verified by the following. Suppose $G_i(A,B)$ is a nontrivial set in $\mathbb{C}$,  
which means $(\sum_{j\not= i}|a_{i,j}|)^2+(\sum_{j\not= i}|b_{i,j}|)^2<|a_{i,i}|^2+|b_{i,i}|^2$. 
This is true only if $\sum_{j\not= i}|a_{i,j}|<|a_{i,i}|$ or $\sum_{j\not= i}|b_{i,j}|<|b_{i,i}|$, so  the $i$th row of at least one of  $A$ and $B$ has to be strictly diagonally dominant. Hence, $\Gamma_i(A,B)$ is a nontrivial subset of $\mathbb{C}$. 

To see the contrary is not true,  consider the pencil 
\begin{equation}\label{example1}
A_1-\lambda B_1=\begin{pmatrix}2&3\\3&2\end{pmatrix}-\lambda \begin{pmatrix}2&1\\1&2\end{pmatrix},\end{equation}
which has eigenvalues  $-1$ and $5/3$.  $\Gamma(A_1,B_1)$ for this pencil is 
$\Gamma(A_1,B_1)=\left\{z \in \mathbb{C}:\left|z-1\right|\leq 4 \right\}$. 
 In contrast, $G(A_1,B_1)$ is $G(A_1,B_1)= \left\{z \in \mathbb{C}:\chi(\lambda, 1)\leq  \sqrt{10/8}\right\},$
 which is useless because the chordal radius is larger than 1. 
\item 
When $B\simeq I$, $\Gamma(A,B)$ is always a tighter region than $G(A,B)$ is, because $G_i(A,I)$ is
\[\frac{|\lambda-a_{i,i}|}{\sqrt{1+|\lambda|^2}\sqrt{1+|a_{i,i}|^2}}\lesssim \sqrt{\frac{\left(\sum_{j\not= i}|a_{i,j}|\right)^2}{1+|a_{i,i}|^2}}= \sqrt{\frac{R_i^2}{1+|a_{i,i}|^2}}.\]
\[\therefore |\lambda-a_{i,i}|\lesssim \sqrt{1+|\lambda|^2}R_i,\]
 whereas $\Gamma_i^S(A,I)$ is the standard Gerschgorin set
\[ |\lambda-a_{i,i}|\leq  R_i,\]
 from which $\Gamma_i^S(A,B)\subseteq G_i(A,I)$ follows trivially.
\end{itemize}
 \subsection{A  tighter result}\label{gersec2}
Here we show that we can obtain a slightly tighter eigenvalue inclusion set by bounding the center of the disk \eqref{inz} more carefully. Instead of Lemma \ref{rbound},  we use the following two results.
\begin{lemma}\label{rbound2}
The point $1/(1+re^{i\theta})$ where $r\geq 0$ and $\theta \in[0,2\pi]$ lies on a circle of radius $r/(1-r^2)$ centered at $1/(1-r^2)$.
\end{lemma}
\begin{proof}
\begin{align*}
\left|\frac{1}{1+re^{i\theta}}-\frac{1}{1-r^2}\right|&=\left|\frac{(1-r^2)-(1+re^{i\theta})}{(1+re^{i\theta})(1-r^2)}\right|\\
&=\left|\frac{r(r+e^{i\theta})}{(1+re^{i\theta})(1-r^2)}\right|\\
&=\left|\frac{re^{i\theta}(1+re^{-i\theta})}{(1+re^{i\theta})(1-r^2)}\right|\\
&=\left|\frac{r}{1-r^2}\right|.\quad (\because \left|\frac{1+re^{-i\theta}}{1+re^{i\theta}}\right|=1)
\end{align*}
\end{proof}
\begin{lemma}\label{rbound3}
 Denote by  $M(r)$ the disk of radius  $r/(1-r^2)$ centered at  $1/(1-r^2)$. 
If  $0\leq r'<r<1$ then $M(r')\subseteq M(r)$.
\end{lemma}
\begin{proof}
We prove by showing that $z\in M(r')\Rightarrow z\in M(r)$. Suppose $z\in M(r')$. $z$ satisfies $\left|z-\frac{1}{1-(r')^2}\right|\leq \left|\frac{r'}{1-(r')^2}\right|$, so
\begin{align*}
\left|z-\frac{1}{1-r^2}\right|
&\leq\left|z-\frac{1}{1-(r')^2}\right|+\left|\frac{1}{1-(r')^2}-\frac{1}{1-r^2}\right|\\
&\leq  \left|\frac{r'}{1-(r')^2}\right|+\left|\frac{r^2-(r')^2}{(1-(r')^2)(1-r^2)}\right|\\
&= \frac{r'(1-r^2)+r^2-(r')^2}{(1-(r')^2)(1-r^2)}.
\end{align*}
Here, the right-hand side is smaller than $r/(1-r^2)$, because
\begin{align*}
\frac{r}{1-r^2}- \frac{r'(1-r^2)+r^2-(r')^2}{(1-(r')^2)(1-r^2)}
&= \frac{r(1-(r')^2)-(r'(1-r^2)+r^2-(r')^2)}{(1-(r')^2)(1-r^2)}\\
&= \frac{(1-r)(1-r')(r-r')}{(1-(r')^2)(1-r^2)}>0.
\end{align*}
Hence $\left|z-\frac{1}{1-r^2}\right|\leq \frac{r}{1-r^2}$,  so $z\in M(r)$. Since the above argument holds for any $z\in M(r')$,  $M(r')\subseteq M(r)$ is proved. 
\end{proof}

The implication of these two Lemmas applied to  \eqref{inz} is that the center $s/(1+\gamma)$ lies in $sM(r)$. Therefore we  conclude that $z$ that satisfies \eqref{inz} is included in the disk centered at $\displaystyle \frac{s}{1-r^2}$, and of radius $\displaystyle\frac{sr}{1-r^2}+\frac{t}{1-r}$. 

Therefore, it follows  that  $\lambda$ that satisfies \eqref{kok2} lies in the disk of radius $\displaystyle\frac{|a_{i,i}|r_i+R_i(1+r_i)}{|b_{i,i}|(1-r_i^2)}$,  centered at $\displaystyle\frac{a_{i,i}}{b_{i,i}(1-r_i^2)}$.



Similarly, we can conclude that $1/\lambda$ that satisfies \eqref{kok2a} has to satisfy
\begin{equation}\label{kok2a2ti}
 \left|\frac{1}{\lambda} -\frac{b_{i,i}}{a_{i,i}}\cdot \frac{1}{1-(r_i^A)^2}\right|\leq  \frac{|b_{i,i}|r_i^A+R_i^A(1+r_i^A)}{|a_{i,i}|(1-(r_i^A)^2)}. 
\end{equation}
 Recalling the analysis that derives \eqref{simple}, we see that when $b_{i,i}\not=0$, this inequality is equivalent to 
 \begin{equation}   \label{simpletight}
|\lambda-\tilde \alpha_i|\leq \tilde\beta_i |\lambda|,   
 \end{equation}
 where $\displaystyle\tilde\alpha_i=a_{i,i}(1-(r_i^A)^2)/b_{i,i}$, $\displaystyle\tilde \beta_i=r_i^A+R_i^A(1+r_i^A)/|b_{i,i}|$. 

The equality  of \eqref{simpletight} holds on an Apollonius circle, whose radius is 
$\displaystyle\tilde\rho^A_i=\frac{|\tilde \alpha_i|\tilde\beta_i}{|1-\tilde\beta_i^2|},$ and  center is $\displaystyle c_i=\frac{\tilde\alpha_i}{1-\tilde\beta_i^2}$.

The above analyses leads to the following definition, analogous to that in Definition \ref{def1}. 
 \begin{definition}\label{def2}
We use the same notations $S ,S^A,r_i,R_i,r_i^A,R_i^A$ as in Definition \ref{def1}. 

For $i\in S^B$, define the disk $\tilde\Gamma_i^B$ by
\begin{equation}\label{mygbtilde}
\tilde \Gamma_i^B(A,B)\equiv\left\{z \in \mathbb{C}:\left|z-\frac{a_{i,i}}{b_{i,i}}\frac{1}{1-(r_i )^2}\right|\leq \tilde\rho_i \right\} 
\quad (i=1, 2, \cdots, n),
\end{equation} 
where the radii $\tilde\rho_i $ are defined by
\[\tilde\rho_i = \frac{|a_{i,i}|r_i+R_i(1+r_i)}{|b_{i,i}|(1-r_i^2)}.\]
For $i\notin S^B$, we set $\tilde\Gamma_i^B(A,B)=\mathbb{C}$. 

 $\tilde\Gamma_i^A(A,B)$ is defined by the following.  For $i\in S^A$ and $b_{i,i}\not= 0$,  denote 
\[\tilde\alpha_i=\frac{a_{i,i}}{b_{i,i}}(1-(r_i^A)^2), \tilde \beta_i=r_i^A+\frac{R_i^A(1+r_i^A)}{|b_{i,i}|},  \tilde c_i=\frac{\alpha_i}{1-\beta_i^2}\ \mbox{and}\  \tilde\rho_i^A=\frac{|\alpha_i|\beta_i}{|1-\beta_i^2|}.\]
 Then, $\tilde\Gamma_i^A(A,B)$ is defined similarly to $\Gamma_i^A(A,B)$ (by replacing $\alpha_i,\beta_i,c_i,\rho_i^A$ with $\tilde\alpha_i,\tilde\beta_i,\tilde c_i,\tilde \rho_i^A$ respectively in \eqref{myg2}-\eqref{myghalf}), depending on whether $\tilde\beta_i>1,\tilde\beta_i<1$ or $\tilde\beta_i=1$. 

When $b_{i,i}=0$ or $i\notin S^A$, $\tilde\Gamma^A_i(A,B)=\Gamma^A_i(A,B)$ 
defined  in Definition \ref{def1}. 

\end{definition}

Thus we arrive at  a slightly tighter Gerschgorin theorem.
\begin{theorem}[Tighter Gerschgorin-type theorem]\label{ger2thm}
 Let $A,B$ be $n\times n$ complex  matrices. 

All the eigenvalues of the pencil $A-\lambda B$  lie in the  union of $n$ regions $\tilde\Gamma_i(A,B)$ in the complex plane defined by
\begin{equation}\label{mygtilde}
\tilde \Gamma_i(A,B)\equiv  \tilde\Gamma_i^B(A,B)\cap \tilde\Gamma_i^A(A,B).
\end{equation} 
 In other words, if $\lambda$ is an eigenvalue of the pencil, then
\[\lambda\in \tilde\Gamma(A,B)\equiv\bigcup_{1\leq i\leq n}\tilde\Gamma_i(A,B).\]
\end{theorem}
 The  proof is the same as the one  for Theorem \ref{gger} and is omitted.  The simplified results of Theorem \ref{ger2thm} analogous to Corollaries \ref{simp} and \ref{simp2} can also be  derived but is  omitted.

 It is easy to see that $\tilde \Gamma_i(A,B)\subseteq \Gamma_i(A,B)$ for all $i$, so $\tilde \Gamma(A,B)$ is a sharper eigenvalue bound than $\Gamma(A,B)$. 
For example, for the pencil \eqref{example1}, we have $\tilde \Gamma_i(A_1,B_1)=\left\{z \in \mathbb{C}:\left|z-\frac{4}{3}\right|\leq \frac{11}{3} \right\}$. 
We can also  see that  $\tilde \Gamma(A,B)$ shares all the properties mentioned at the end of  section \ref{gersec1}. 
The reason  we presented $\Gamma(A,B)$ although $\tilde \Gamma(A,B)$ is always tighter is that 
  $\Gamma(A,B)$  has centers $a_{i,i}/b_{i,i}$, which may make it simpler to apply than $\tilde \Gamma(A,B)$. In fact, in the analysis in  section \ref{forerr} we  only use Theorem \ref{gger}. 

 \subsection{Localizing a specific number of eigenvalues}\label{gersec3}
We are sometimes interested not only in the eigenvalue bounds, but also  in the number of eigenvalues included in a certain region. The classical Gerschgorin theorem serves this need \cite{varga}, which has the property  that if a region contains exactly $k$ Gerschgorin disks and is disjoint from the other disks, then it contains exactly $m$ eigenvalues. 
This fact is used  to derive a perturbation result for simple eigenvalues in \cite{wilkinson:1965}. An analogous result holds for the set $G(A,B)$ \cite[Ch.5]{stewart-sun:1990}. 
Here we show that our Gerschgorin set also possesses the same property.

\begin{theorem}\label{disjoint} If a union of $k$ Gerschgorin regions $\Gamma_i(A,B)$ (or $\tilde \Gamma_i(A,B)$) in the above Theorems (Theorem \ref{gger}, \ref{ger2thm} or Corollary \ref{simp}, \ref{simp2})  is disjoint from the remaining $n-k$ regions and is not the entire complex plane $\mathbb{C}$, then exactly $k$ eigenvalues of the pencil $A-\lambda B$ lie in the union.
\end{theorem}
\begin{proof}

We prove the result for $\Gamma_i(A,B)$. The other sets can be treated in an entirely identical way. 

 We use the same trick used for proving the analogous result for the set $G(A,B)$, shown in \cite[Ch.5]{stewart-sun:1990}. 
Let  $\tilde A=\mbox{diag}(a_{11},a_{22},\cdots,a_{nn}), \tilde B=\mbox{diag}(b_{11},b_{22},\cdots,b_{nn})$ and define
\[A(t)=\tilde A+t(A-\tilde A), \quad B(t)=\tilde B+t(B-\tilde B).\]
It is easy to see that the Gerschgorin disks $\Gamma_i(A(t),B(t))$ get enlarged as $t$ increases from $0$ to $1$. 

 In \cite{stewart-sun:1990} it is shown in the chordal metric that the eigenvalues of a regular  pencil $A-\lambda B$ are continuous functions of the elements provided that the pencil is regular. 

 Note that each of the regions $\Gamma_i(A(t),B(t))$ is a closed and bounded subset of $\mathbb{C}$ in the chordal metric, and that if a union of $k$ regions $\Gamma_i(A(t),B(t))$ is disjoint from the other $n-k$ regions in the Euclidean metric, then this disjointness holds also in the chordal metric. 
Therefore, if the  pencil $A(t)-\lambda B(t)$ is regular for $0\leq t\leq 1$, then an eigenvalue that is included in a certain union of $k$ disks $\bigcup_{1\leq i\leq k}\Gamma_i(A(t),B(t))$ cannot jump to another disjoint region as $t$ increases, so the claim is proved. Hence it suffices to prove that the  pencil $A(t)-\lambda B(t)$ is regular. 

 The regularity is proved by contradiction.  If $A(t)-\lambda B(t)$ is singular for some $0\leq t\leq 1$, then any point  $z\in \mathbb{C}$ is an eigenvalue of the pencil. However, the disjointness assumption implies that 
 there must exist a point $z'\in \mathbb{C}$ such that $z'$ lies in none of the Gerschgorin disks, so $z'$ cannot be an eigenvalue. Therefore, $A(t)-\lambda B(t)$ is necessarily regular for  $0\leq t\leq 1$. 

\end{proof}
\section{Examples}\label{ex}
Here we show some examples to illustrate the regions we discussed above. As test matrices we consider the simple pencil $A-\lambda B\in \mathbb{C}^{n\times n}$ where
\begin{equation}
  \label{testmat}
A=\begin{pmatrix}4&a&&\\a&4&\ddots&\\&\ddots&\ddots&a\\&&a&4\end{pmatrix}\quad \mbox{and}\quad B=\begin{pmatrix}4&b&&\\b&4&\ddots&\\&\ddots&\ddots&b\\&&b&4\end{pmatrix}.  
\end{equation}
Note that $\Gamma(A,B)$, $\tilde \Gamma(A,B)$ and $K(A,B)$ are nontrivial regions  if $b<2$, and $G(A,B)$ is nontrivial only if $a^2+b^2<8$.
Figure \ref{geronly} shows our results  $\Gamma(A,B)$ and $\tilde \Gamma(A,B)$ for different parameters $(a,b)$. The two crossed points indicate the smallest and largest eigenvalues of the pencil \eqref{testmat} when the matrix size is $n=100$. For $(a,b)=(1,2)$ the largest eigenvalue (not shown) was $\simeq 1034$. Note that $\Gamma(A,B)=\Gamma^B(A,B)$ when $(a,b)=(2,1)$ and $\Gamma(A,B)=\Gamma^A(A,B)$ when $(a,b)=(1,2)$. 
\begin{figure}[htbp]
\begin{tabular}{cc}
\begin{minipage}{0.33\hsize}
\begin{center}
\includegraphics[width=47mm]{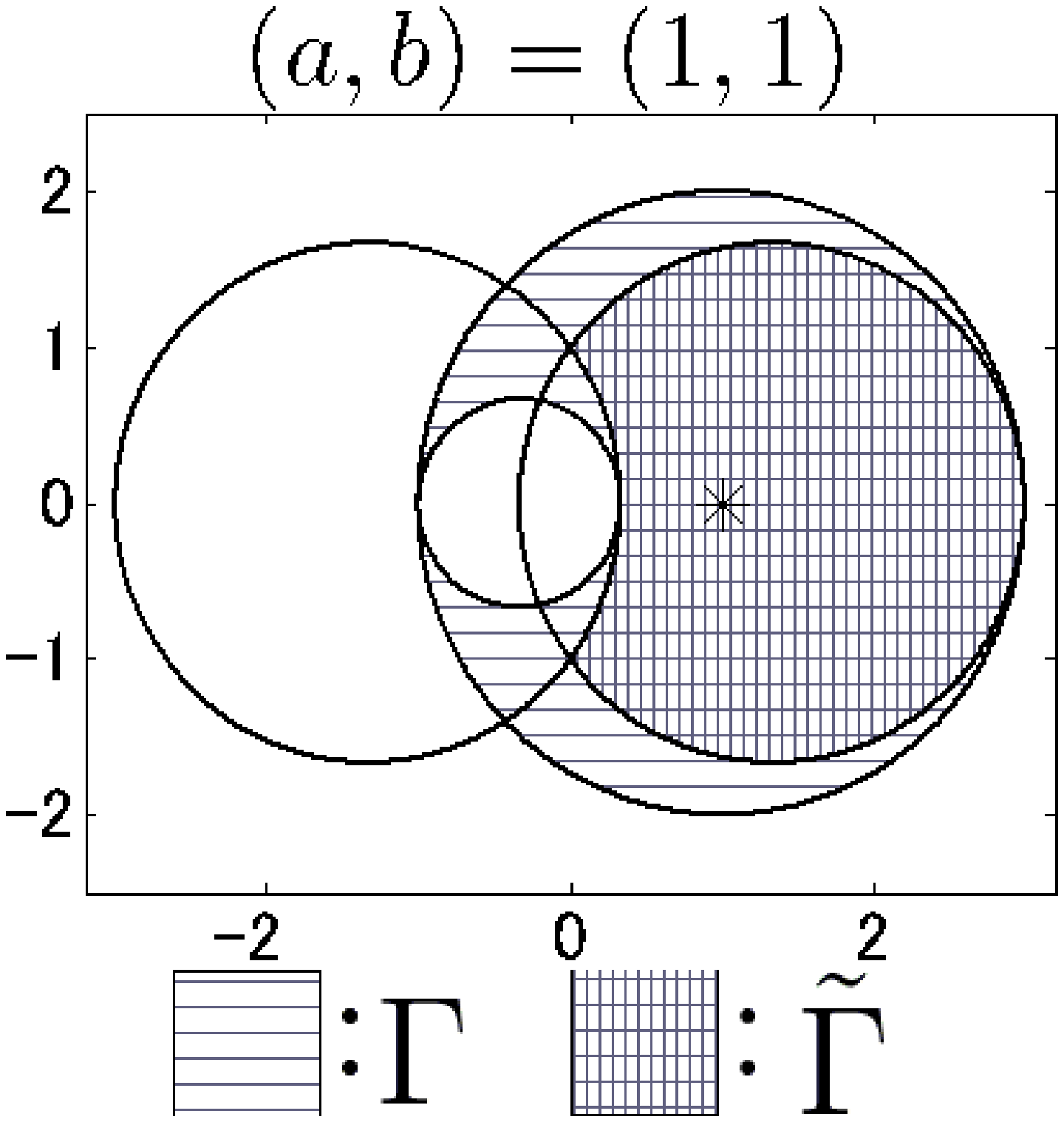}
\end{center}
\end{minipage}
\begin{minipage}{0.33\hsize}
\begin{center}
\includegraphics[width=47mm]{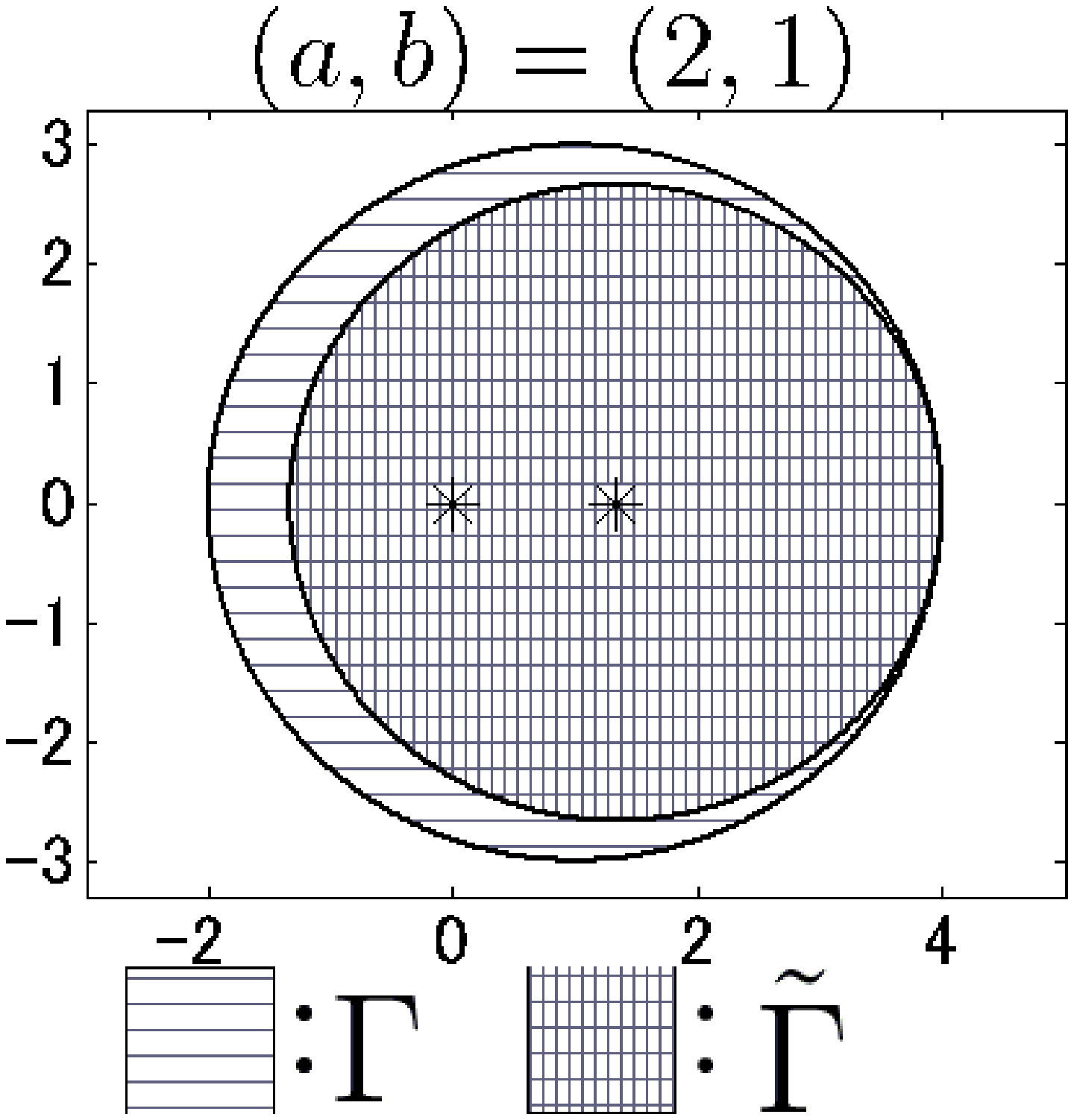}
\end{center}
\end{minipage}
\begin{minipage}{0.33\hsize}
\begin{center}
\includegraphics[width=47mm]{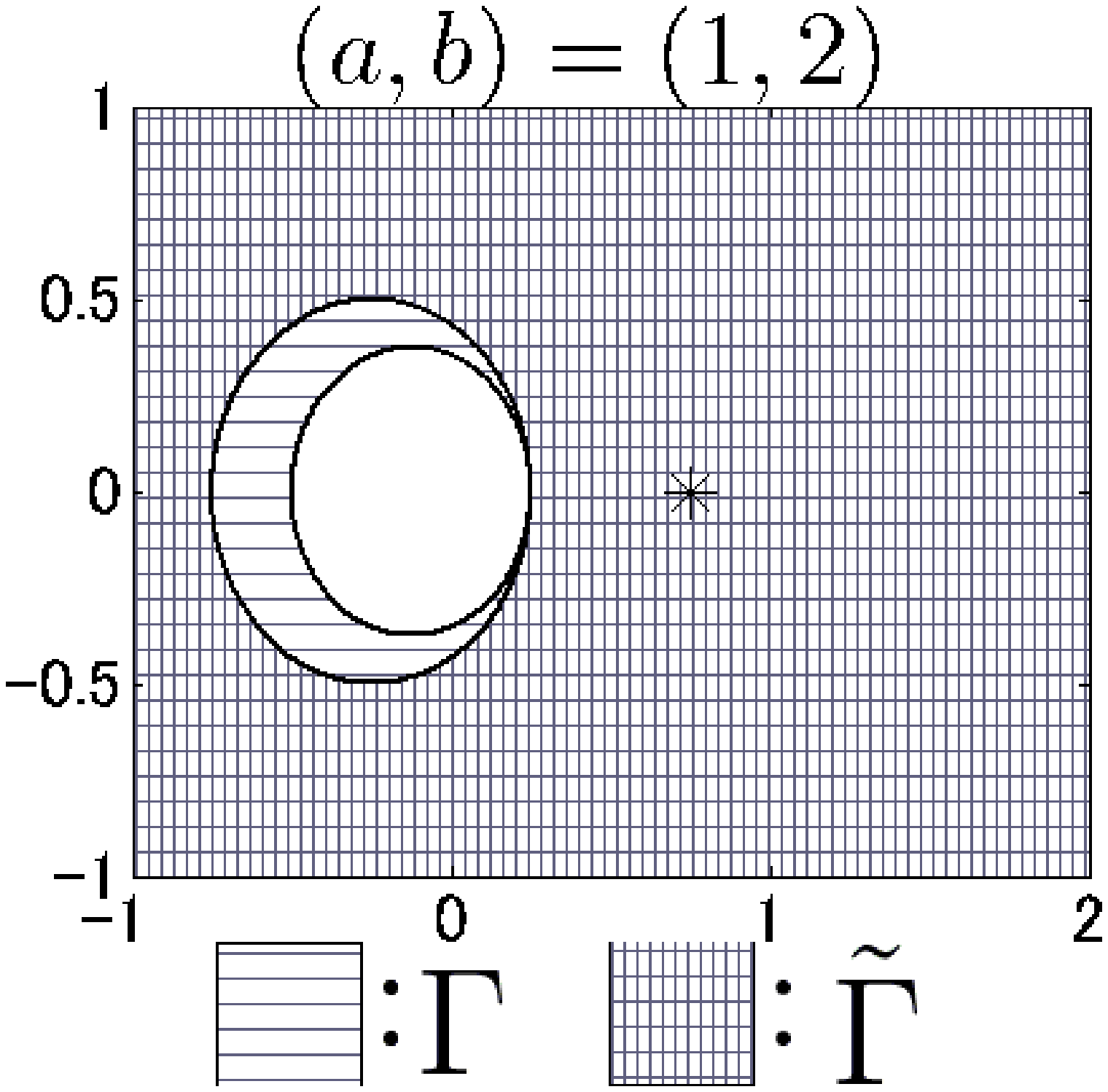}
\end{center}
\end{minipage}
\end{tabular}
\caption{Plots of $\Gamma(A,B)$ and $\tilde\Gamma(A,B)$ for matrices \eqref{testmat} with different $a,b$.}
\label{geronly}
\end{figure}

The purpose of the figures below is to compare our results with the known results $G(A,B)$ and $K(A,B)$. As for our results we only show $\Gamma^S(A,B)$ for simplicity.  Figure \ref{gerandste} compares $\Gamma^S(A,B)$ with $G(A,B)$. We observe that in the cases $(a,b)=(2,1),(3,1)$,  $\Gamma(A,B)$ is a much more useful set than $G(A,B)$ is, which in the latter case is the whole complex plane. This reflects the observation given in section \ref{gersec1} that $\Gamma(A,B)$ is always tighter when $B\simeq I$. 
\begin{figure}[htbp]
\begin{tabular}{cc}
\begin{minipage}{0.33\hsize}
\begin{center}
\includegraphics[width=47mm]{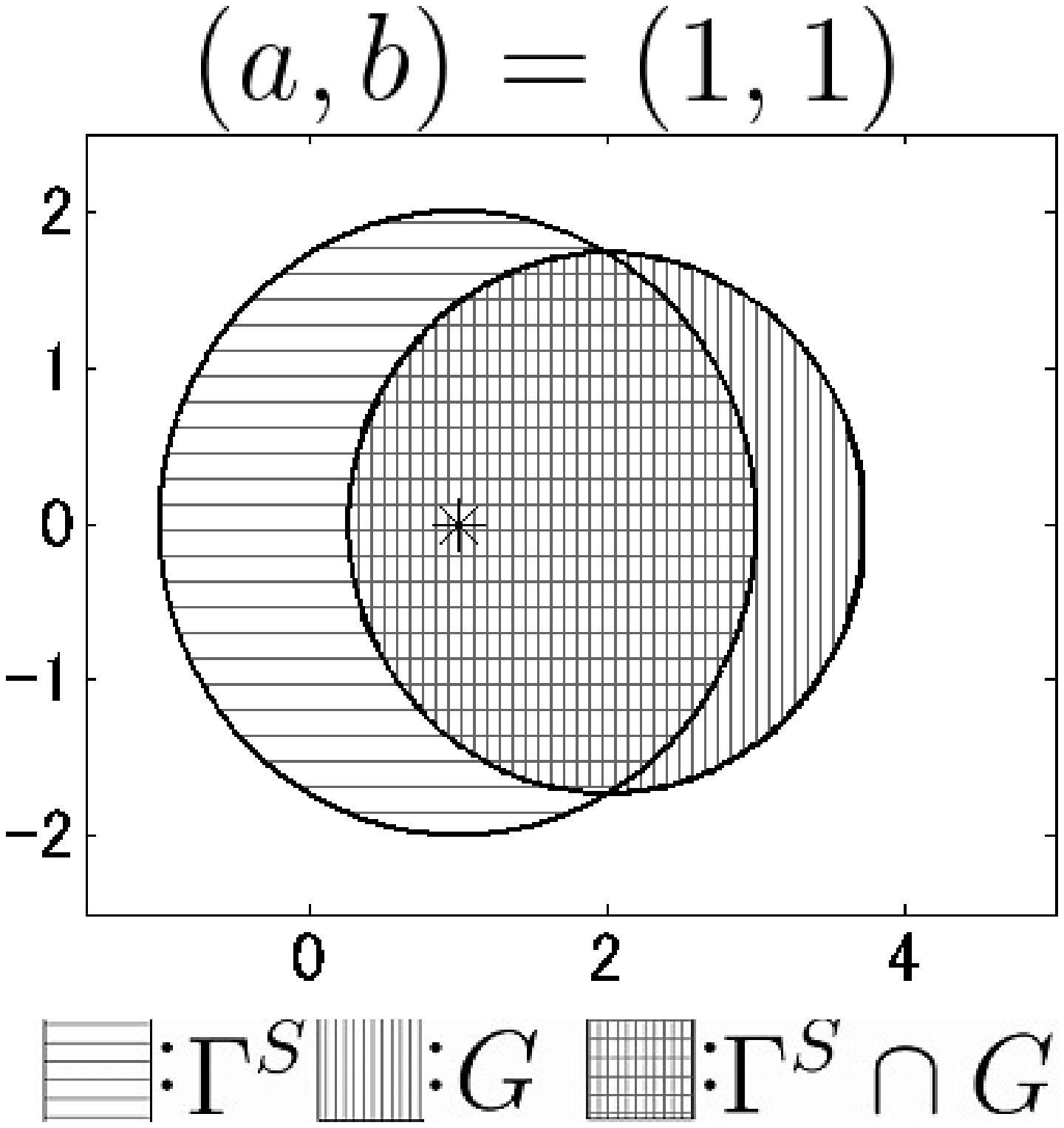}
\end{center}
\end{minipage}
\begin{minipage}{0.33\hsize}
\begin{center}
\includegraphics[width=47mm]{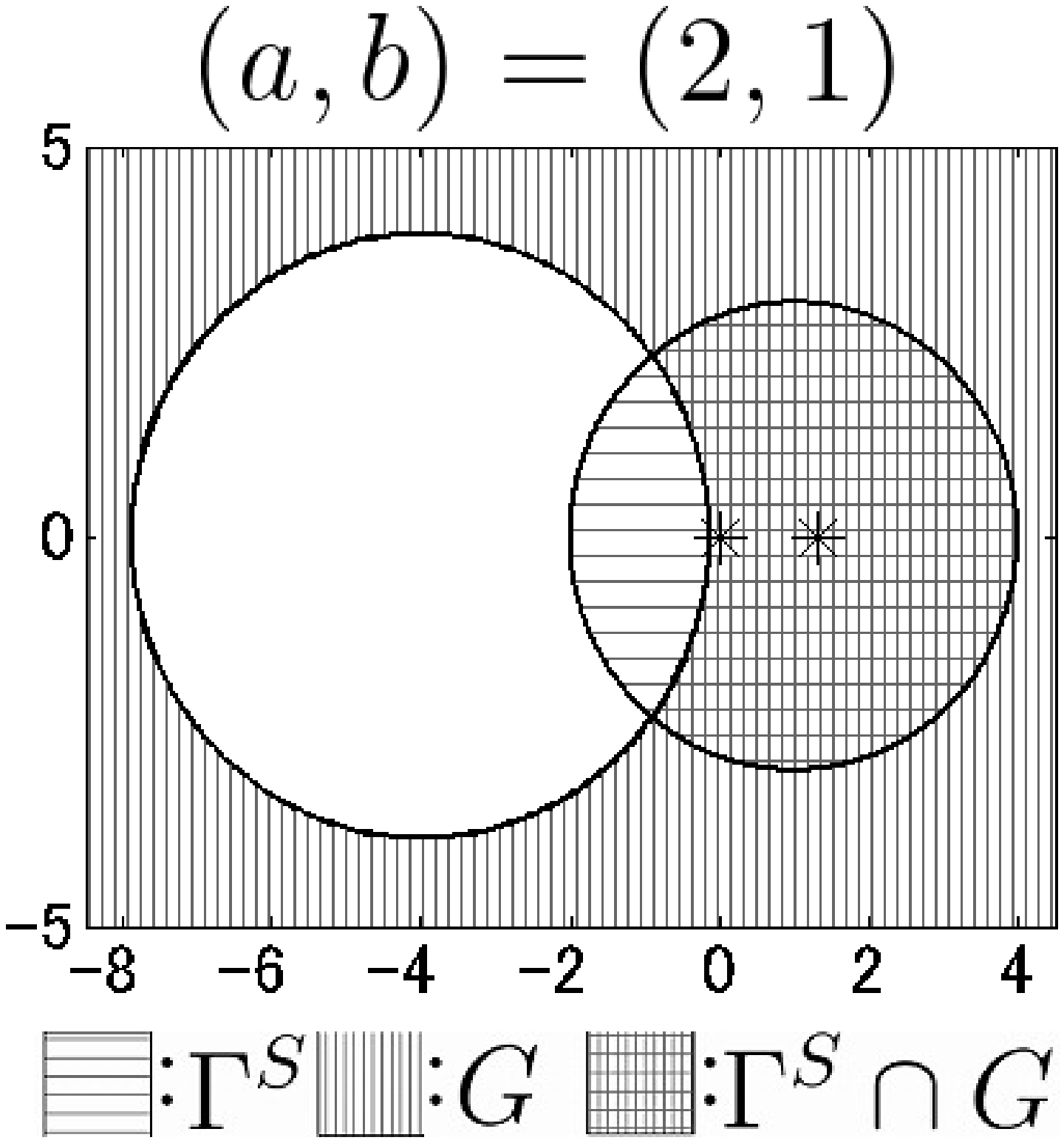}
\end{center}
\end{minipage}
\begin{minipage}{0.33\hsize}
\begin{center}
\includegraphics[width=47mm]{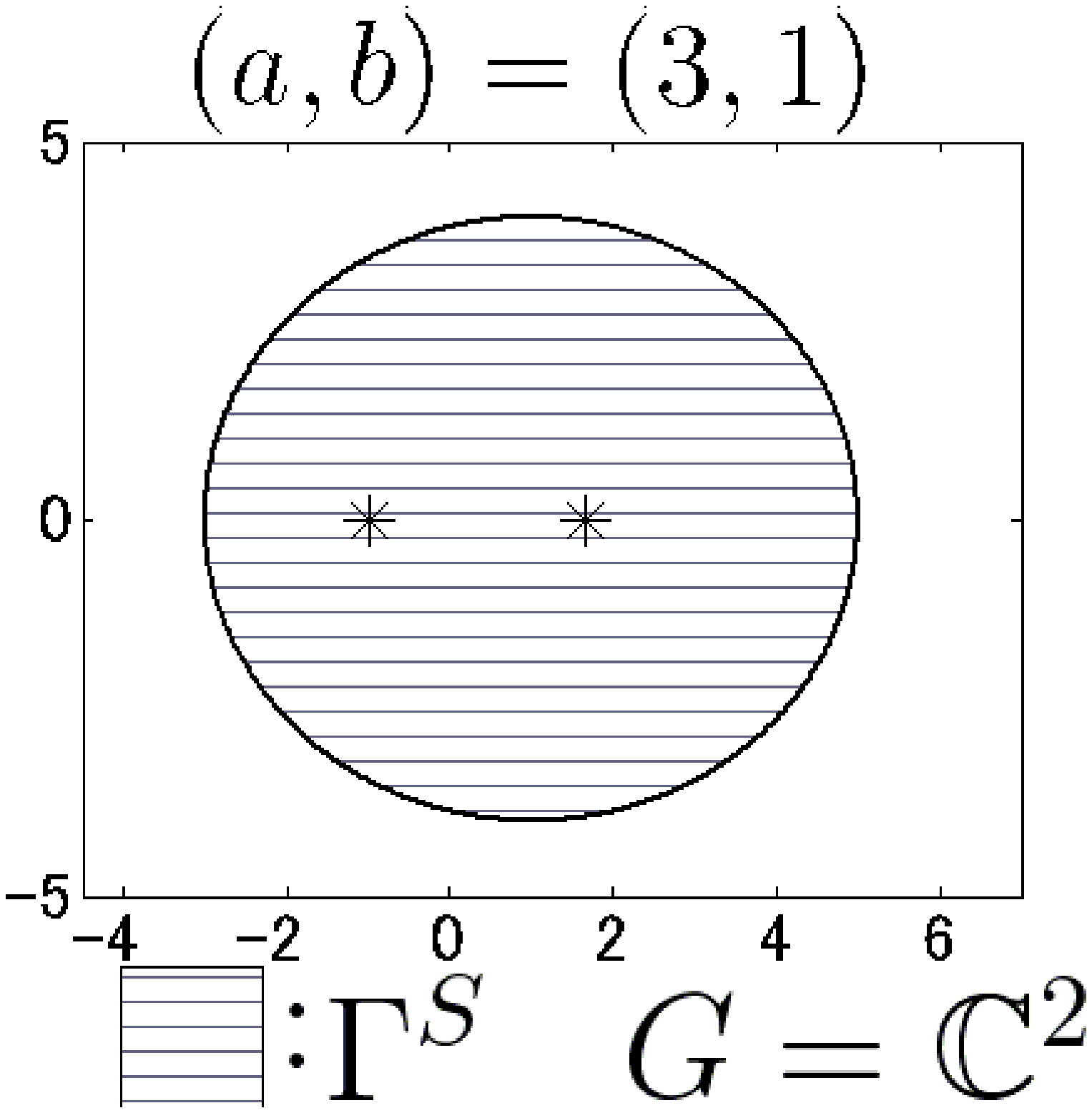}
\end{center}
\end{minipage}
\end{tabular}
\caption{Plots of $\Gamma^S(A,B)$ and $G(A,B)$ for matrices \eqref{testmat} with different $a,b$.}
\label{gerandste}
\end{figure}

 Figure \ref{gerandvar} compares $\Gamma^S(A,B)$ with $K(A,B)$, in which the boundary of $\tilde\Gamma^S(A,B)$ is shown as dashed circles. We verify the relation $K(A,B)\subseteq \tilde\Gamma(A,B)\subseteq\Gamma(A,B)$. These three sets become equivalent when $B$ is nearly diagonal, as shown in  the middle graph. 
The right graph shows the regions for the matrix defined in Example 1 in \cite{gengervarga}, in which all the eigenvalues are shown as crossed points.
\begin{figure}[htbp]
\begin{tabular}{cc}
\begin{minipage}{0.33\hsize}
\begin{center}
\includegraphics[width=47mm]{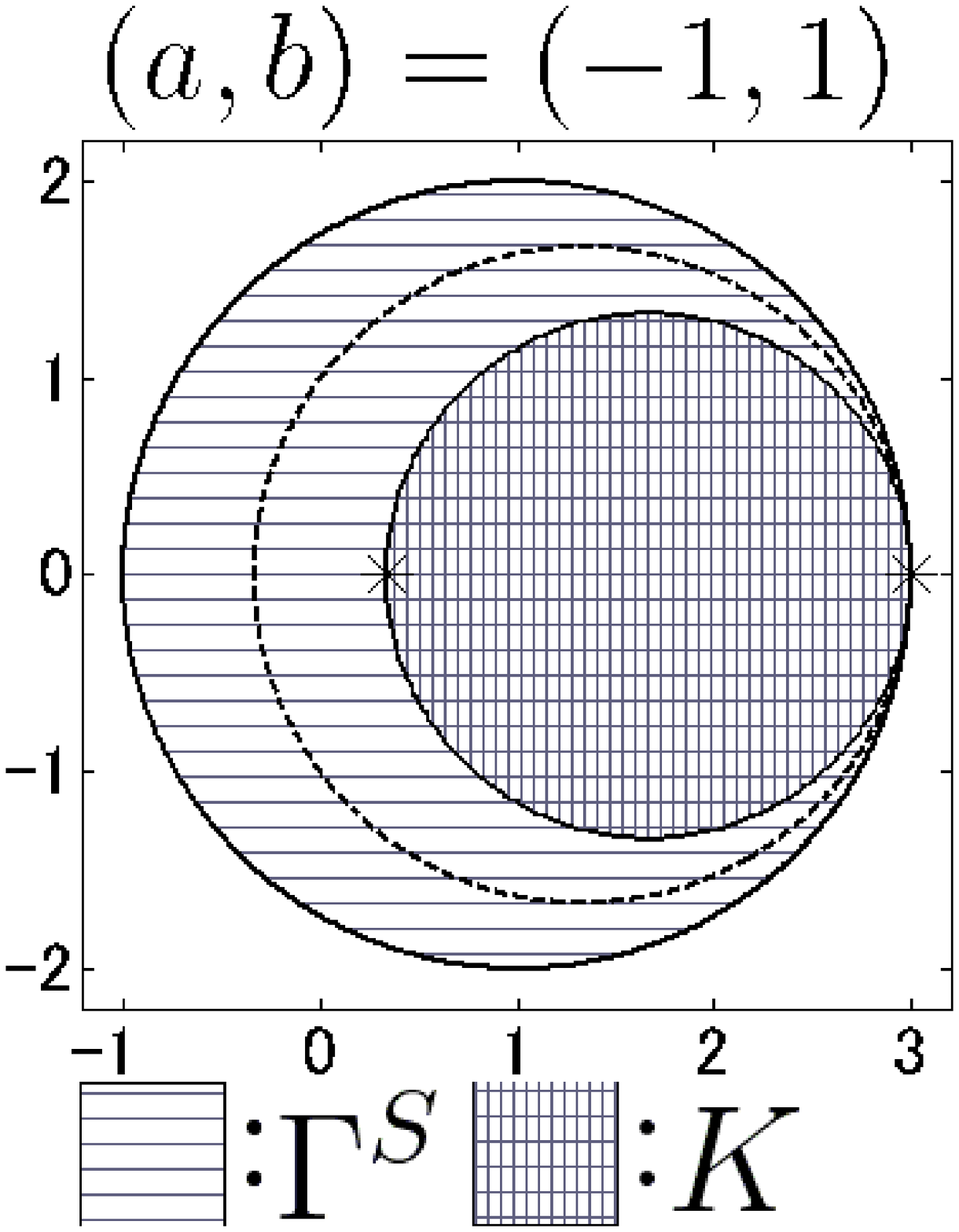}
\end{center}
\end{minipage}
\begin{minipage}{0.33\hsize}
\begin{center}
\includegraphics[width=47mm]{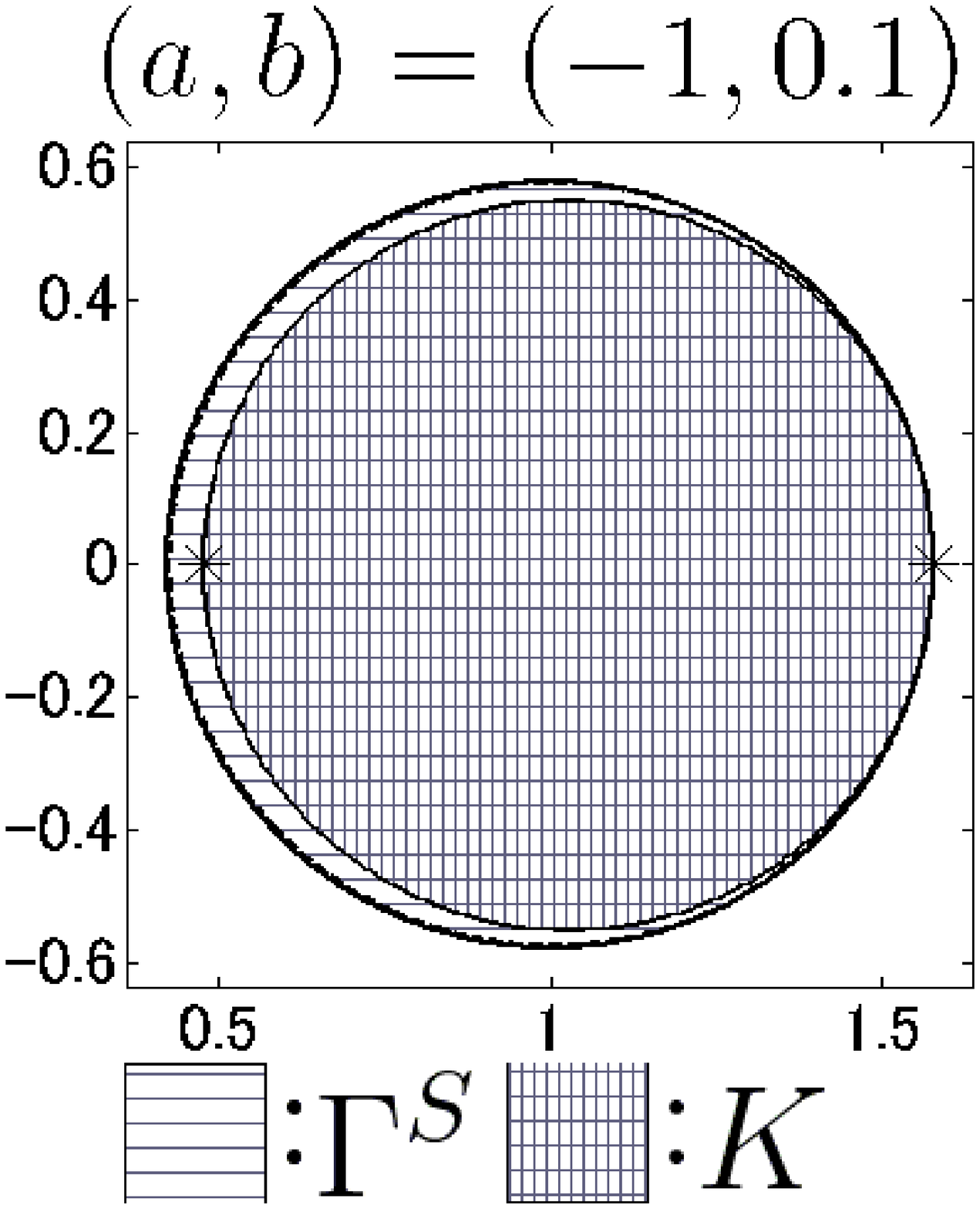}
\end{center}
\end{minipage}
\begin{minipage}{0.33\hsize}
\begin{center}
\includegraphics[width=47mm]{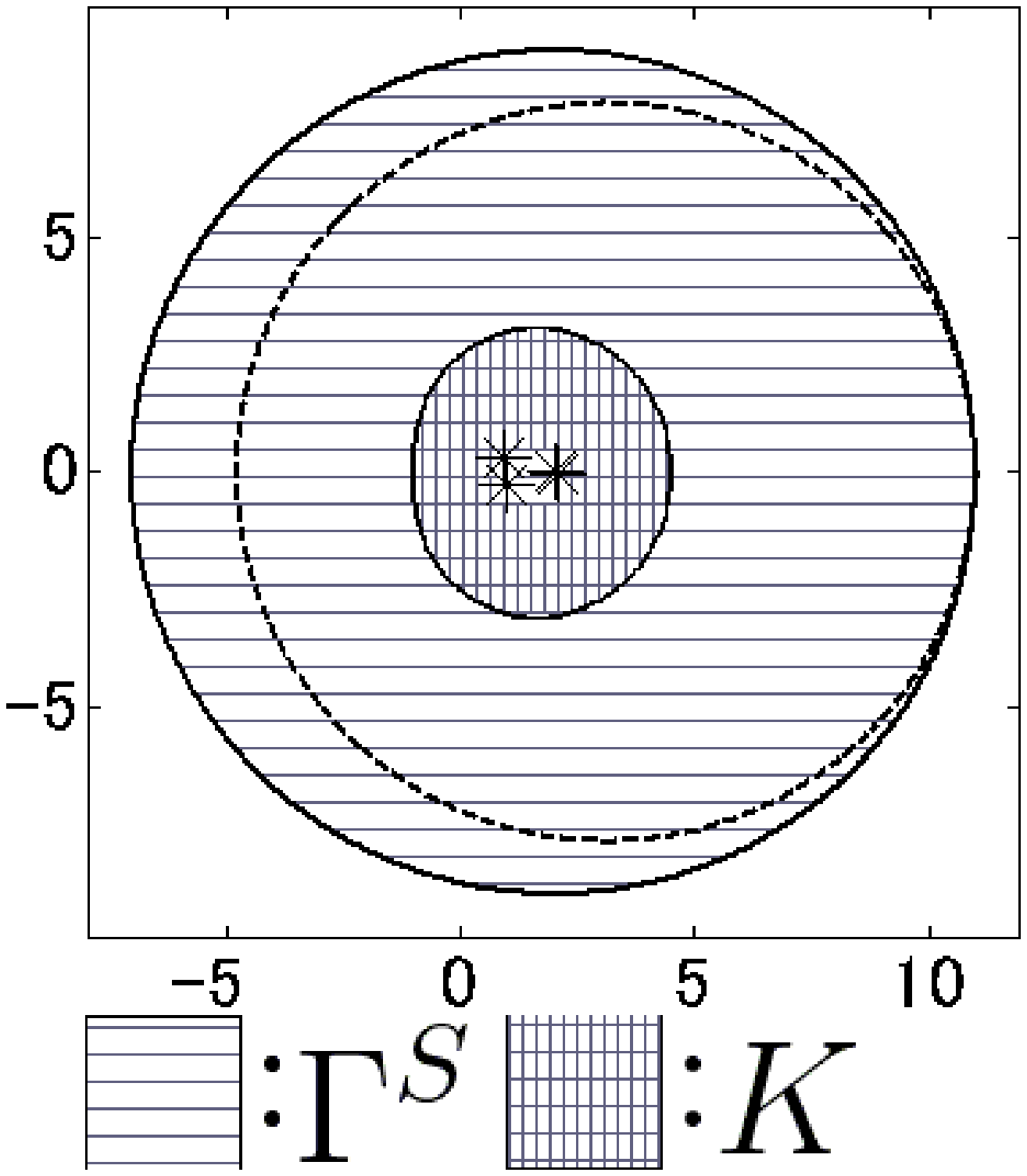}
\end{center}
\end{minipage}
\end{tabular}
\caption{Plots of $\Gamma(A,B)$ and $K(A,B)$}
\label{gerandvar}
\end{figure}

\newpage 
 We emphasize that our result $\Gamma(A,B)$ is defined by circles and so is easy to plot, while the regions $K(A,B)$ and $G(A,B)$ are generally complicated regions, and are difficult to plot. In the above figures  we obtained $K(A,B)$ and $G(A,B)$  by a very naive method, i.e., by dividing the complex plane into small regions and testing whether the center of the region is contained in each set. 
\section{Application to forward error analysis}\label{forerr}
The Gerschgorin theorems presented in section \ref{ggsec} can be used in a straightforward way for a matrix pencil with some diagonal dominance property whenever one wants a simple estimate for the eigenvalues or bounds for  the extremal eigenvalues, as the standard Gerschgorin theorem is used for standard eigenvalue problems. 

 Here we show how our results can also be used to provide  a forward error analysis for computed eigenvalues of a diagonalizable pencil $A-\lambda B\in \mathbb{C}^{n\times n}$. 

 For simplicity we assume only finite eigenvalues exist.  After the  computation of eigenvalues 
  $\tilde \lambda_i\ (1\leq i\leq n)$ and  eigenvectors (both left and right) 
one can normalize the eigenvectors to get $X,Y\in \mathbb{C}^{n\times n}$ such that 
\begin{align*}
Y^HAX(\equiv \widehat A)&=\begin{pmatrix}\tilde \lambda_1&e_{1,2}&\cdots&e_{1,n}\\e_{2,1} &\tilde\lambda_2&\ddots&\vdots\\\vdots &\ddots&\ddots&e_{n-1,n}\\e_{n,1}&\cdots&e_{n,n-1}&\tilde\lambda_n\end{pmatrix}=\mbox{diag}\{\tilde \lambda_1,\cdots,\tilde \lambda_n\}+E,\\
  Y^HBX(\equiv \widehat B)&=\begin{pmatrix} 1&f_{1,2}&\cdots&f_{1,n}\\f_{2,1} &1&\ddots&\vdots\\\vdots &\ddots&\ddots&f_{n-1,n}\\f_{n,1}&\cdots&f_{n,n-1}&1\end{pmatrix}=I+F.\end{align*}
 The  matrices $E$ and $F$ represent the errors, which we expect to be small after a successful computation (note that in practice computing the matrix products $Y^HAX, Y^HBX$ also introduces errors, but here we ignore this effect, to focus on the accuracy of the eigensolver). 
 We denote by  $E_{j}=\sum_{l}|e_{j,l}|$ and $F_{j}=\sum_{l}|f_{j,l}|$ ($1\leq j\leq n$) their absolute $j$th row sums. 
 We assume that $F_j<1$ for all $j$, or equivalently that  $I+F$ is strictly diagonally dominant, so in the following we only consider  $\Gamma^S(A,B)$ in Corollary \ref{simp} and refer to it  as the Gerschgorin disk. 

 We note that the assumption that both eigenvalues and eigenvectors are computed restricts the problem size to moderate $n$. Nonetheless, computation of a full eigendecomposition of a small-sized problem is often necessary   in practice. For example, it is the computational kernel of the Rayleigh-Ritz process in a method for computing several eigenpairs of a large-scale problem \cite[Ch. 5]{baitemplates}. 



\subsection*{Simple bound}
For a particular computed eigenvalue $\tilde \lambda_i$, we are interested in   how close it is to an ``exact'' eigenvalue of the pencil $A-\lambda B$. 
We consider the simple and multiple eigenvalue cases separately. 
\begin{enumerate}
\item When $\tilde \lambda_i$ is a simple eigenvalue. We  define $\delta\equiv \min_{j\not= i}|\tilde\lambda_i-\tilde \lambda_j|>0$. If $E$ and $F$ are small enough, then $\Gamma_i(\widehat A,\widehat B)$ is disjoint from all the other $n-1$ disks. 
Specifically, this is true if   $\delta>\rho_i+\rho_j$ for all $j\not= i$,
 where 
 \begin{equation}
   \label{rhoj}
\rho_i=\frac{|\tilde \lambda_i|F_i+E_i}{1-F_i},\quad \rho_j=\frac{|\tilde \lambda_i|F_j+E_j}{1-F_j}   
 \end{equation}
are  the radii of the $i$th and $j$th Gerschgorin disks in Theorem \ref{gger}, respectively.
If the inequalities are satisfied for all $j\not=i$, then using Theorem \ref{disjoint} we conclude that there exists exactly $1$ eigenvalue $\lambda_i$ of the pencil $\widehat A-\lambda \widehat B$ (which has the same eigenvalues as $A-\lambda B$) such that 
\begin{equation}\label{bound1}
|\lambda_i-\tilde \lambda_i|\leq \rho_i. 
\end{equation}

\item When $\tilde \lambda_i$ is a multiple eigenvalue of multiplicity $k$, so that $\tilde \lambda_i=\tilde\lambda_{i+1}=\cdots =\tilde\lambda_{i+k-1}$. 
It is straightforward to see that a similar argument holds and  if the $k$ disks $\Gamma_{i+l}(\widehat A,\widehat B)$ $(0\leq l\leq k-1)$ are disjoint from the other $n-k$ disks, 
then there exist exactly $k$ eigenvalues $\lambda_j\ (i\leq j\leq i+k-1)$ of the pencil $A-\lambda B$ such that 
 \begin{equation}\label{mul1}
 |\lambda_{j}-\tilde \lambda_{i}|\leq \max_{0\leq l\leq k}\rho_{i+l}.
\end{equation}
\end{enumerate}
\subsection*{Tighter bound}
 Here we derive another bound that can be  much tighter than \eqref{bound1} when the error matrices $E$ and $F$ are small. 
 We use  the technique of diagonal similarity transformations employed in \cite{wilkinson:1965,stewart-sun:1990}, where first-order eigenvalue perturbation results are obtained. 

 We consider the case where $\tilde \lambda_i$ is a simple eigenvalue and denote $\delta\equiv \min_{j\not= i}|\tilde\lambda_i-\tilde \lambda_j|>0$, and suppose that the $i$th Gerschgorin disk of the pencil $\widehat A-\lambda \widehat B$ is disjoint from the others. 

Let $T$ be a diagonal matrix whose $i$th diagonal is $\tau$ and $1$ otherwise. We consider the Gerschgorin disks  $\Gamma_j(T\widehat AT^{-1}, T\widehat BT^{-1})$, and find the smallest  $\tau$ such that the $i$th disk is disjoint from the others. By the assumption, this disjointness holds when $\tau=1$, so we only consider $\tau<1$. 

The center of  $\Gamma_j(T\widehat AT^{-1}, T\widehat BT^{-1})$ is $\tilde \lambda_j$ for all $j$. As for the radii  $\widehat \rho_i$ and   $\widehat \rho_j$, for $\tau<F_i,F_j$ we have 
\[ \widehat \rho_i=\frac{\tau|\tilde \lambda_i| F_i+\tau E_i}{1-\tau F_i}\leq \tau \rho_i,\]
 and
 \[\widehat \rho_j\leq \frac{|\tilde \lambda_j| F_j/\tau+E_{j}/\tau}{1-F_j/\tau},\quad   \mbox{for}\  j\not= i. \]
Since $\tau <1$, we see that writing  $\delta_j=|\tilde \lambda_i-\tilde \lambda_j|$,
\begin{equation}\label{jouken}
\rho_i+\widehat \rho_j< \delta_j
\end{equation}
 is a sufficient condition to for the disks $\Gamma_i(T\widehat AT^{-1}, T\widehat BT^{-1})$ 
and $\Gamma_j(T\widehat AT^{-1}, T\widehat BT^{-1})$  to be  disjoint. \eqref{jouken} is satisfied if 
\begin{align*}
&\rho_i+\frac{|\tilde \lambda_j| F_j/\tau+ E_j/\tau}{1-F_j/\tau}<\delta_j\\
\Leftrightarrow &(\tau-F_j)(\delta_j-\rho_i)>|\tilde \lambda_j| F_j+ E_j\\
\Leftrightarrow &\tau>F_j+\frac{|\tilde \lambda_j| F_j+ E_j}{\delta_j-\rho_i},
\end{align*}
where we used $\delta_j-\rho_i>0$, which follows from the disjointness assumption. 
Here, since $\delta_j\geq \delta>\rho_i$, we see that \eqref{jouken} is true if 
\[\tau>F_j+\frac{|\tilde \lambda_j| F_j+ E_j}{\delta-\rho_i}.\]
Repeating the same argument for all $j\not= i$, we conclude that if 
\begin{equation}
  \label{tau0}
\tau>F_j+\frac{\max_{j\not= i}\{|\tilde \lambda_j| F_j+ E_j\}}{\delta-\rho_i}\ (\equiv \tau_0),
\end{equation}
then the disk $\Gamma_i(T\widehat AT^{-1}, T\widehat BT^{-1})$ is disjoint from the remaining $n-1$ disks. 

Therefore, by letting  $\tau=\tau_0$ and using Theorem \ref{disjoint} for the pencil $T\widehat AT^{-1}-\lambda T\widehat BT^{-1}$,  we conclude that there exists exactly one eigenvalue $\lambda_i$ of the pencil $A-\lambda B$ such that 
\begin{equation}\label{bound2}
|\lambda_i-\tilde \lambda_i|\leq \frac{\tau_0(|\tilde \lambda_i|F_i+E_i)}{1-\tau_0F_i}\leq \tau_0\rho_i.
\end{equation}


 Using $\delta\leq|\tilde \lambda_i|+|\tilde \lambda_j|$, we can bound  $\tau_0$ from above by 
\[\tau_0\leq \frac{\max_{j\not= i}\{(2|\tilde \lambda_j|+|\tilde \lambda_i|) F_j+ E_j\}}{\delta-\rho_i}\leq  \frac{\max_{j\not= i}\{(2|\tilde \lambda_j|+|\tilde \lambda_i|) F_j+ E_j\}}{(1-F_i)(\delta-\rho_i)}.\]
 Also observe from  \eqref{rhoj} that 
\[\rho_i=\frac{|\tilde \lambda_i|F_i+E_i}{1-F_i}\leq  \frac{\max_{1\leq j\leq n}\{(2|\tilde \lambda_j|+|\tilde \lambda_i|) F_j+ E_j\}}{1-F_i}.\]
Therefore, denoting $\delta'=\delta-\rho_i\quad \mbox{and}\quad r=\frac{1}{1-F_i}\max_{1\leq j\leq n}\{(2|\tilde \lambda_j|+|\tilde \lambda_i|)F_j+E_j\}$,
we have  $\tau_0\leq r/\delta'$ and $\rho_i\leq r$. Hence, from \eqref{bound2} we conclude that 
\begin{equation}\label{boundquad}
|\lambda_i-\tilde \lambda_i|\leq \frac{r^2}{\delta'}.
\end{equation}
Since $r$ is essentially the size of the error, and $\delta'$ is essentially the gap between $\tilde\lambda_i$ and any other computed eigenvalue, we note that this bound resembles the quadratic bound for the standard Hermitian eigenvalue problem,  $|\tilde \lambda-\lambda|\leq \|R\|^2/\delta$ \cite[Ch.11]{parlettsym}. 
Our result  \eqref{boundquad} indicates that  this type of quadratic error bound holds also for the non-Hermitian generalized eigenvalue problems.

\providecommand{\bysame}{\leavevmode\hbox to3em{\hrulefill}\thinspace}
\providecommand{\MR}{\relax\ifhmode\unskip\space\fi MR }
\providecommand{\MRhref}[2]{%
  \href{http://www.ams.org/mathscinet-getitem?mr=#1}{#2}
}
\providecommand{\href}[2]{#2}


\begin{thebibliography}{10}

\bibitem{baitemplates}
Zhaojun Bai, James Demmel, Jack Dongarra, Axel Ruhe, and Henk van~der Vorst,
  \emph{Templates for the solution of algebraic eigenvalue problems: a
  practical guide}, SIAM, Philadelphia, USA, 2000.

\bibitem{xiao07}
X.~S. Chen, \emph{{On perturbation bounds of generalized eigenvalues for
  diagonalizable pairs}}, Numer. Math. \textbf{{107}} ({2007}), no.~{1},
  {79--86} ({English}).

\bibitem{Golubbook}
Gene~H. Golub and Charles~F. Van~Loan, \emph{Matrix computations}, {The Johns
  Hopkins University Press}, 1996.

\bibitem{gengervarga}
V.~Kostic, L.~J. Cvetkovic, and R.~S. Varga, \emph{{Gersgorin-type
  localizations of generalized eigenvalues}}, Numer. Linear Algebr. Appl.
  \textbf{{16}} ({2009}), no.~{11-12}, {883--898}.

\bibitem{li02}
Ren-Cang Li, \emph{On perturbations of matrix pencils with real spectra, a
  revisit}, Math. Comput. \textbf{72} (2002), no.~242, 715--728.

\bibitem{geometry}
Stanley~C. Ogilvy, \emph{Excursions in geometry}, Dover, 1990.

\bibitem{parlettsym}
B.~N. Parlett, \emph{The symmetric eigenvalue problem}, SIAM, Philadelphia,
  1998.

\bibitem{gersch}
S.~$\rm{Ger\check{s}gorin}$, \emph{{$\rm \ddot U$ber die Abgrenzung der
  Eigenwerte einer Matrix}}, Izv. Akad. Nauk SSSR Ser. Mat. 1 \textbf{7}
  ({1931}), {749--755}.

\bibitem{Stewart75}
G.~W. Stewart, \emph{{Gershgorin} theory for the generalized eigenvalue problem
  {$Ax= \lambda Bx$}}, Math. Comput \textbf{29} (1975), no.~130, 600--606.

\bibitem{stewart-sun:1990}
G.~W. Stewart and J.-G Sun, \emph{Matrix perturbation theory}, Academic Press,
  1990.

\bibitem{varga}
R.~S. Varga, \emph{${Ger\check{s}gorin}$ and his circles}, Springer-Verlag,
  2004.

\bibitem{wilkinson:1965}
J.~H. Wilkinson, \emph{The algebraic eigenvalue problem (numerical mathematics
  and scientific computation)}, Oxford University Press, USA, April 1965.

\end{thebibliography}
\end{document}